\documentclass{amsart}
\pdfoutput=1 

\usepackage[all,hyperref,numberbysection]{bi-discrete}

\usepackage[backend=biber,maxbibnames=5,maxalphanames=5,style=alphabetic,bibencoding=utf8,giveninits,url=false,isbn=false]{biblatex}

\newcommand{\fint}{\dashint}

\providecommand{\eps}{{\epsilon}}

\AtBeginBibliography{\small}

\begin{document} 

\title{Global regularity for nonlinear systems with symmetric gradients}%

\author{Linus Behn and Lars Diening}
\address{Linus Behn \\
	Fakult\"{a}t  f\"{u}r Mathematik,
	University of Bielefeld\\
	Universit\"{a}tsstrasse 25, 33615 Bielefeld, Germany}
\email{linus.behn@math.uni-bielefeld.de}

\address{Lars Diening \\
	Fakult\"{a}t  f\"{u}r Mathematik,
	University of Bielefeld\\
	Universit\"{a}tsstrasse 25, 33615 Bielefeld, Germany}
\email{lars.diening@uni-bielefeld.de}

\thanks{The authors gratefully acknowledge financial support by the German Research Foundation through the IRTG 2235 (Project 282638148).}%



\begin{abstract}
  We study global regularity of nonlinear systems of partial differential equations  depending on the symmetric part of the gradient with Dirichlet boundary conditions. These systems arise from variational problems in plasticity with power growth. We cover the full range of exponents $p \in (1,\infty)$. As a novelty  the degenerate case for $p>2$ is included.
  We present a unified approach for all exponents by showing the regularity for general systems of Orlicz growth.
\end{abstract}

\maketitle

\section{Introduction}
\label{sec:introduction}

In this paper we investigate the global regularity of solutions of nonlinear elliptic systems arising from problems in plasticity, see e.g.~\cite{FuchsSer00}. In particular, we are interested in systems of the form
\begin{equation}\label{eq:system}
	\begin{aligned}
		-\divergence \big((\delta+|\eps u|)^{p-2}\eps u\big)&= f\qquad \text{in }\Omega \\
		u&= 0 \qquad \text{on }\partial \Omega .		
	\end{aligned}
\end{equation}
Here, $\Omega \subset \RRn$ is a bounded domain, $p\in(1,\infty)$, $\delta \geq 0$, $u:\Omega \rightarrow \RRn$ is a vector field and $\eps u\coloneqq \tfrac 12 (\nabla u + (\nabla u)^T)$ is the symmetric part of $\nabla u$, also abbreviated as \emph{symmetric gradient}. 

Similar systems arise in the context of non-Newtonian fluids (power-law fluids), where additionally a pressure and a incompressibility constraint appears, see e.g.~\cite{MalekNecasRokytaRuzicka1996}.

For $\delta = 0$ our system is also known as  the \emph{symmetric $p$-Laplace system}. Sometimes in the literature the case $\delta=0$ and $1<p<2$ is called \emph{singular} and the case $\delta=0$ and $p>2$ is called \emph{degenerate}.

The system~\eqref{eq:system} arises as the Euler-Lagrange equation of the functional
\begin{align}
  \label{eq:defJphi}
  J_\phi(u)\coloneqq  \int_\Omega \phi(|\eps u|) \,dx - \int_\Omega f u\,dx,
\end{align}
where $\phi$ is the Young function
\begin{align*}
  \phi(t) &= \int_0^t (\delta+s)^{p-2} s\,ds.
\end{align*}
For $\delta=0$ we obtain $\phi(t) = \frac 1p t^p$. In fact, although the system~\eqref{eq:system} is our main interest, we also consider systems arising from more general Young functions~$\phi$.  Our technique has the advantage that we do not have to distinguish in our proofs the cases $p<2$ and $p>2$ as is often the case in the literature.

In the study of our system~\eqref{eq:system} two quantities are of special interest
\begin{align*}
  A(\epsilon u) &\coloneqq  (\delta+\abs{\epsilon u})^{p-2} \epsilon u,
  \\
  V(\epsilon u) &\coloneqq  (\delta+\abs{\epsilon u})^{\frac{p-2}{2}} \epsilon u.
\end{align*}
The former $A(\epsilon u)$ is sometimes called the \emph{stress tensor}. The latter $V(\epsilon u)$ appears naturally in the study of regularity of nonlinear systems with power growth. 
Note that $A$ and $V$ fulfill the relation $A(\epsilon u) : \epsilon u = \abs{V(\epsilon u)}^2$.
Our goal is to establish global regularity properties for solutions $u$ of~\eqref{eq:system}. This is best expressed in terms of~$V(\epsilon u)$. In particular, we will show that $V(\epsilon u) \in W^{1,2}(\Omega)$ up to the boundary. This is the natural regularity when applying the technique of difference quotients which formally corresponds to testing the equation with~$\Delta u$.

We start by giving an overview of already existing results. Let us start with the simpler case of the standard $p$-Laplace system, which depends on the gradient~$\nabla u$ rather than the symmetric gradient~$\eps u$. There the proof of $V(\nabla u) \in W^{1,2}_{\loc}(\Omega)$ is rather classical and can be found e.g. in \cite[Theorem~1.11.1]{Morrey2008}, \cite[Lemma~3.1]{Uhl77} and \cite[Proposition~2.4]{AceFus89} or in~\cite[Theorem~11]{DieEtt08} for Orlicz growth. Global regularity~$V(\nabla u) \in W^{1,2}(\Omega)$ have been established in~\cite{CheDiB89}.

Let us now discuss the known results on regularity for our system~\eqref{eq:system} with symmetric gradients. Interior regularity $V(\epsilon u) \in W^{1,2}_{\loc}(\Omega)$ for $A(\epsilon u)$ roughly of the form $A(\epsilon u) = (1+\abs{\epsilon u})^{p-2} \epsilon u$ with $p \in (1,2)$ has been shown in~\cite{Seregin1992}. The technique using difference quotients is similar to the one for the case of dependence on~$\nabla u$. Only in the case of global regularity $V(\epsilon u) \in W^{1,2}(\Omega)$ the question becomes significantly more difficult.
The first result on global regularity is due to Seregin and Shilkin~\cite{SerShi97} who proved the global regularity $V(\eps u)\in W^{1,2}(\Omega)$ for the case of a flat boundary and for operators of the form $A(\eps u)\coloneqq (\delta +|\eps u| )^{p-2}\eps u$, where $1<p\leq 2$ and $\delta >0$. This result was extended by Berselli and \Ruzicka{}~\cite{BerRuz20} to non-flat boundaries $\partial \Omega$ of class $C^{2,1}$ with $1 < p \leq 2$ and $\delta\geq 0$. The case  $p \geq 2$ but with $\delta >0$ (non-degenerate case) has been proved in~\cite{BerRuz22b}, where also the parabolic situation is investigated. The estimates depend strongly on~$\delta>0$ and deteriorate for~$\delta \to 0$. In this paper we will prove $V(\eps u)\in W^{1,2}(\Omega)$ for all $1<p<\infty$ and $\delta \geq 0$. To include all the possible variants of growth conditions we work in the even more general setting of Orlicz growth. For the special case of $p$-growth our main result is the following Theorem~\ref{thm:p-growth}. The general case can be found in Theorem~\ref{thm:orlicz}.
\begin{theorem}\label{thm:p-growth}
  Let $\Omega$ be a bounded $C^{2,1}$-domain, $1<p<\infty$, $\delta \geq 0$ and $f\in W_0^{1,p^\prime}(\Omega)$. Then the system
  \begin{equation*}
    \begin{aligned}
      -\divergence \big((\delta+|\eps u|)^{p-2}\eps u\big) &=f \qquad \text{in $\Omega$},\\
      u&=0 \qquad\text{on $\partial \Omega$}
    \end{aligned}
  \end{equation*}
  has a unique weak solution $u$ fulfilling $V(\eps u)\coloneqq (\delta+|\eps u|)^{\tfrac{p-2}{2}}\eps u\in W^{1,2}(\Omega)$ with the estimate
  \begin{align*}
    \| V(\eps u)\|^2_{W^{1,2}(\Omega)}\lesssim  \|f\|_{W^{1,p^\prime}}^{p^\prime}.
  \end{align*}
  If $p\leq 2$, this implies $u\in W^{2,\tfrac{np}{n+p-2}}(\Omega)$, and if $p\geq 2$, it implies $u\in W^{1+\tfrac{2}{p},p}(\Omega)$.
\end{theorem}

Our assumption of a $C^{2,1}$-regular boundary is the same that is made in~\cite{BerRuz20}. It is a technical assumption that we need because of the way we locally flatten the boundary. In particular, we straighten the boundary locally only in one direction. A more flexible method would allow slightly less regular boundaries.

Also our assumptions on $f$ are in some cases more stringent than necessary. This is due to the fact that $\phi$ is a general N-function and does not need to have $p$ growth for one specific $p$. If it was, e.g., of $p$ growth with $1<p\leq2$, then $f\in L^{p'}(\Omega)$ would suffice. However, for $p>2$ it is not sufficient to assume $f\in L^{p'}$, see~\cite[Section 5]{BraSan18}. It might be possible to reduce $f \in W^{1,p'}$ for $p > 2$ to $W^{s,p'}$ for some $s \in (\frac{p-2}{p},1)$ as done in~\cite{BraSan18}.  The condition that $f$ vanishes at the boundary is needed in order to perform integration by parts without getting any boundary terms and is important for~$p \geq 2$.

Let us make a few additional remarks where the regularity theory of the symmetric $p$-Laplace system lacks behind the one for the $p$-Laplace system involving~$\nabla u$ instead of~$\epsilon u$.  Higher H\"older regularity $u \in C^{1,\alpha}$, in other words $V(\nabla u) \in C^{0,\beta}$, has been shown in~\cite{Uhl77,AceFus89} for the $p$-Laplace system and in~\cite{DieningStroffoliniVerde09} for the corresponding Orlicz growth. Such regularity up to the boundary has been proved in~\cite{CheDiB89} even for parabolic problems.
Such $C^{1,\alpha}$-regularity results are completely open for the symmetric $p$-Laplacian for $n \geq 3$. The case~$n=2$ has been solved in~\cite{KaplickyMalekStara1999,DieningKaplickySchwarzacher2014}, where even the $p$-Stokes system is covered.

Another aspect of regularity is the one of the stress~$A(\nabla u)$ in the case of the $p$-Laplace system.  If $f \in L^{n,1}(\Omega)$ (Lorentz space), then $A(\nabla u) \in L^\infty(\Omega)$, resp $\nabla u \in L^\infty(\Omega)$, has been proved in \cite{CianchiMazya2014bounded,KuusiMingione2018}. If $f = \divergence F$ and $F \in C^{0,\alpha}(\overline{\Omega})$, then $A(\nabla u) \in C^{0,\beta}(\overline{\Omega})$ has been proved in~\cite{BreitCianchiDieningSchwarzacher2022}. 
The maximal regularity $A(\nabla u) \in W^{1,2}(\Omega)$ for the $p$-Laplace system with $f \in L^2(\Omega)$ and $p > 2(2-\sqrt{2}) = 1.171532\dots$ has been shown in~\cite{CiMa_JMPA,BalCiaDieMaz21}. The corresponding result $A(\epsilon u) \in W^{1,2}(\Omega)$ or even $A(\epsilon u) \in W^{1,2}_{\loc}(\Omega)$ is unsolved for the symmetric $p$-Laplace system. The only known result on regularity transfer from the data to the stress can be found in~\cite{DieningKaplicky2013}. There it is shown that if $F \in L^q(\Omega)$ with $p \leq q < 3p$ and $n=3$, then $A(\epsilon u) \in L^q_{\loc}(\Omega)$. This is even shown for the $p$-Stokes system.

Let us make a few final remarks on the related $p$-Stokes system, where an additional pressure and incompressibility constraint appear. Here $V(\epsilon u)\in W^{1,2}(\Omega)$ is only known for systems with quadratic growth. In the language of N-functions this means $\phi''(t) \eqsim 1$, see e.g. \cite{FuchsSer00}. For $1<p\leq 2$ and $\delta>0$ one has $V(\epsilon u) \in W^{1,q}(\Omega)$ for some $q <2$, see~\cite{BerRuz17}.

The rest of the paper is organized as follows. In Section~\ref{sec:object-main-result} we introduce uniformly convex N-functions and formulate our main result in the language of Orlicz spaces. In Section~\ref{sec:basic-prop-weak} we derive Korn's and Caccioppoli's inequalities in our setting. In Section~\ref{sec:global-regularity} we prove the main result. We do this by showing it first for N-functions with quadratic growth (Section~\ref{sec:bound-regul-line}) and then use an approximation argument to prove the general case (Section~\ref{sec:passage-limit}). In the Appendix~\ref{sec:unif-conv-orlicz-appendix} we collect some elementary properties of uniformly convex N-functions that are needed throughout the proofs.

\section{Objective and main result}
\label{sec:object-main-result}

Our main focus is the study of system~\eqref{eq:system}. This fits nicely in the more general framework of systems with Orlicz growth.  In this section we introduce the corresponding partial differential system and the necessary background on Orlicz spaces and N-functions.

\subsection{Uniformly convex Orlicz spaces}
\label{sec:unif-conv-orlicz}

We begin with a bit of standard notation. With $a\lesssim b$ we mean that there exists a constant $c>0$ independent of all important quantities such that $a\leq cb$. We use the notation $a\lor b\coloneqq \max (a,b)$ and $a\land b\coloneqq  \min (a,b)$. For a measurable set $U\subset \RRn$ and $f\in L^1_\loc (\RRn)$ we denote by $|U|$ the Lebesgue measure $\mathcal{L}^n(U)$ and define
\begin{equation*}
  \mean{f}_U\coloneqq \fint_U f(x)\,dx\coloneqq  \dfrac{1}{|U|}\int_U f(x)\,dx.
\end{equation*}
For a ball $B \subset \RRn$ let $2B$ denote the ball with the same center but twice the diameter. Usually, we denote the radius of a ball~$B$ by~$r$.

An N-function is a function $\phi : \setR^{\geq 0} \rightarrow \setR^{\geq 0}$ which has a right-continuous, non-decreasing derivative $\phi^\prime   : \setR^{\geq 0} \rightarrow \setR^{\geq 0}$, such that $\phi^\prime(0)=0$, $\phi^\prime (t)>0$ for $t>0$, $\lim\limits_{t\rightarrow \infty}\phi^\prime (t)=\infty$ and $\phi(t)=\int_0^t\phi^\prime(s)\,ds$ for all $t\geq 0$.

For every N-function $\phi$ we define its conjugate N-function $\phi^\ast$ via
\begin{equation*}
	\phi^\ast(s)\coloneqq  \sup_{t\geq 0}(st-\phi(t)).
\end{equation*}
With this definition, $\phi^\ast$ is again an N-function and we have $\phi^{\ast\ast}=\phi$. If $\phi^\prime$ is increasing we also have $(\phi^\ast)^\prime(s)=(\phi ^\prime)^{-1}(s)$ for all $s\geq 0$.

For an N-function $\phi$ and open set $\Omega \subset \RRn$ we define the Orlicz space $L^\phi(\Omega)$ as the set of all measurable functions $v:\Omega \rightarrow \setR$ such that $\|v\|_{L^\phi(\Omega)}<\infty$, where
\begin{align*}
  \norm{v}_{L^\phi(\Omega)}\coloneqq  \inf \Bigset{ \lambda >0: \int_\Omega \phi\big(\lambda^{-1}\abs{v(x)}\big) \,dx\leq 1}.
\end{align*}
We define the Sobolev-Orlicz space $W^{1,\phi}(\Omega)\coloneqq \lbrace v \in L^\phi(\Omega): |\nabla v|\in L^{\phi}(\Omega) \rbrace$ equipped with the norm $\|v\|_{W^{1,\phi}(\Omega)}\coloneqq \|v\|_{L^\phi(\Omega)}+\|\nabla v\|_{L^\phi(\Omega)}$.

The systems that we study require that $\phi$ fulfills the following condition of uniform convexity, see \cite{DieForTom20}.
\begin{definition}[Uniform convexity]\label{def:uniform_convexity}
  An N-function $\phi$ is called uniformly convex, if $\phi\in C^1([0,\infty))\cap C^2((0,\infty))$ and we have
  \begin{align*}
    p^-&\coloneqq \inf_{t>0}\dfrac{\phi^{\prime\prime}(t)t}{\phi^{\prime}(t)}+1 >1 \\
    p^+&\coloneqq \sup_{t>0}\dfrac{\phi^{\prime\prime}(t)t}{\phi^{\prime}(t)}+1 <\infty.
  \end{align*}
  We call $p^-$ the lower and $p^+$ the upper index (of uniform convexity) of $\phi$.
\end{definition}

\begin{example}
  Standard examples of uniformly convex N-functions are
  \begin{enumerate}
  \item $\phi(t)=\tfrac{1}{p}t^p$ with $p\in (1,\infty)$. Then $p^-=p^+=p$. This example motivates the definition of $p^-$ and $p^+$.
  \item $\phi(t)=\int_0^t (\delta+s)^{p-2} s\,ds \eqsim (\delta+t)^{p-2}t^2$ with $p \in (1,\infty)$ and $\delta > 0$. Then $p^-= p \wedge 2$ and $p^+=p \vee 2$.
  \item $\phi(t)=t^p + t^q$ with $p,q \in (1,\infty)$. Then $p^- = p \wedge q$ and $p^+ = p \vee q$.
  \end{enumerate}
\end{example}

\begin{remark}
	An example for a non uniformly convex $\phi$ is given by
	\begin{equation*}
		\phi(t):=\int_0^t\int_0^s (1+\tau^2)^{-\mu/2}\,d\tau \,ds,
	\end{equation*}
	where $\mu \in (1,\infty)$. In this case we have $p^+=2$ but $p^-=1$. The interior regularity of the corresponding Euler-Lagrange equation with the symmetric gradient (and $f=0$) is studied in \cite{GmeinKristensen19,Gmein20}.
\end{remark}

Uniform convexity can also be defined for slightly less regular N-functions with $\phi' \in W^{1,\infty}_{\loc}((0,\infty))$ instead of~$C^2((0,\infty))$, see \cite[Definition 28]{DieForTom20}. We collected some properties of uniformly convex N-functions in the Appendix, Section~\ref{sec:unif-conv-orlicz-appendix}.

\subsection{Main result}
\label{sec:main-result}

In this section we introduce our system and present our main result on global regularity.

We denote by $\partial_i$ the partial derivative and by $\nabla = (\partial_1, \dots, \partial_n)$ the gradient of either a scalar function, a vector field or a tensor field. If $A:\RRn \rightarrow \RRnn$ is a function then by $\divergence A:\RRn \rightarrow \RRn$ we mean the row-wise divergence of $A$, i.e. $\divergence A = \sum_{k=1}^n \partial_k A_{j,k}$.  For matrices $M,N\in \RRnn$ we define the Kronecker product $A:B\coloneqq \sum_{i,j=1}^n A_{ij}B_{ij}$ and for vectors $u$, $v$ we set $u\otimes v\coloneqq (u_iv_j)_{ij}$.

For a vector field $u:\RRn \rightarrow \RRn$ we define the symmetric gradient $\eps u$ via
\begin{equation*}
  \eps u \coloneqq  \tfrac{1}{2}\big(\nabla u + (\nabla u)^T\big).
\end{equation*}
In the following let $\phi$ be a fixed uniformly convex N-function with indices of uniform convexity $p^-$ and $p^+$. We define the two important quantities $A$ and $V$ as
\begin{align}
  \label{eq:defAV}
  A(P)\coloneqq \phi^\prime(|P|)\dfrac{P}{|P|}\qquad\text{and}\qquad V(P)\coloneqq \sqrt{\phi^\prime(|P|)|P|}\dfrac{P}{|P|},
\end{align}
with $A(P) = V(P)=0$ for $P=0$.  The quantities $A$, $V$ and the N-function $\phi$ are closely related by the following formula
\begin{align}\label{eq:VA_basic1}
  A(P) : P &= \abs{V(P)}^2 \eqsim \phi(\abs{P}).
\end{align}
Moreover, we have 
\begin{align}\label{eq:VA_basic2}
  (A(P)-A(Q)):(P-Q)&\eqsim |V(P)-V(Q)|^2 \eqsim \phi^{\prime \prime}(|P|+|Q|)|P-Q|^2.
\end{align}
In \eqref{eq:VA_basic1} and \eqref{eq:VA_basic2} all implicit constants depend only on $p^-$ and $p^+$. These and other useful equivalences can be found in the Appendix, Section~\ref{sec:unif-conv-orlicz-appendix}.

Our main result is the following.
\begin{theorem}\label{thm:orlicz}
  Let $\Omega \subset \RRn$ be an open and bounded domain with $C^{2,1}$-regular boundary. Let $\phi$ be a uniformly convex N-function and $A$ and $V$ defined as in~\eqref{eq:defAV}. Then for every $f\in W^{1,\phi^\ast}_0(\Omega)$ the weak solution $u\in W^{1,\phi}(\Omega)$ to the system of equations
  \begin{equation} \label{eq:system_in_main_thm}
    \begin{aligned}
      -\divergence\big(A(\epsilon u)\big)&=f\qquad \text{in }\Omega,\\
      u&=0 \qquad \text{on }\partial \Omega
    \end{aligned}
  \end{equation}
  satisfies $V(\eps u)\in W^{1,2}(\Omega)$ together with the estimate
  \begin{equation}\label{eq:bound_in_main_thm}
    \|V(\eps u)\|^2_{W^{1,2}(\Omega)}\leq c\int_\Omega  \phi^\ast (|f|)+\phi^\ast(|\nabla f|)\,dx,
  \end{equation}
  where $c>0$ is a constant depending only on $n$, $\Omega$ and the upper and lower indices of uniform convexity of $\phi$.
\end{theorem}

The system~\eqref{eq:system_in_main_thm} is the Euler-Lagrange system to the energy~\eqref{eq:defJphi}. Note that Theorem~\ref{thm:p-growth} is a special case of Theorem~\ref{thm:orlicz} with $\phi(t) = \int_0^t (\delta +s)^{p-2}s\,ds$, $p \in (1,\infty)$ and $\delta \geq 0$.

\section{Basic properties of weak solutions}
\label{sec:basic-prop-weak}

In this section we begin with the basic properties of weak solutions from existence to Caccioppoli estimates. We also present Korn's inequalities that are needed in the context of the symmetric gradient.

\subsection{Korn's inequality}
\label{sec:korns-inequality}

We say that an N-function $\phi$ fulfills the $\Delta_2$ condition ($\phi \in \Delta_2$) if there exists a constant $c>0$ such that for all $t\geq 0$
\begin{equation}\label{eq:delta2cond}
	\phi(2t)\leq c~ \phi(t).
\end{equation}
We denote by $\Delta_2(\phi)$ the smallest constant $c$ which verifies inequality \eqref{eq:delta2cond}.

We need the following Korn-type inequality in Orlicz spaces that can be found for example in \cite[Theorem 6.13]{DieRuzSchu10}.

\begin{proposition}[Korn inequality]\label{prop:Korn}
  Let $\phi$ be an N-function with $\phi , \phi^*\in \Delta_2$. Then
  for all balls $B\subset \RRn$ and $u\in W^{1,1}(B;\RRn)$ we have
  \begin{equation*}
    \int_B \phi(|\nabla u-\mean{\nabla u}_B|)\,dx \leq c_1 \int_B \phi(|\eps u-\mean{\eps u}_B|)\,dx .
  \end{equation*}
  The constants $c_1$ depends only on $\Delta_2(\phi)$, $\Delta_2(\phi^*)$ and $n$.
\end{proposition}

As a simple consequence we will obtain the following version for partial zero boundary values.

\begin{proposition}[Korn inequality; boundary]\label{prop:KornBoundary}
  Let $\phi$ be an N-function with $\phi , \phi^*\in \Delta_2$. Let $\Omega$ be a bounded Lipschitz domain.  Let $B \subset \RRn$ be a ball with center in $\overline{\Omega}$ such that $\Omega \cap B$ is again a Lipschitz domain. Then
  for all $u\in W^{1,1}(B \cap \Omega;\RRn)$ such that $u$ vanishes on $\partial \Omega \cap \overline{B}$ we have
  \begin{equation*}
    \int_{B \cap \Omega}\phi(|\nabla u|)\,dx\leq c_2 \int_{B \cap \Omega}\phi(|\eps u|)\,dx.
  \end{equation*}
  The constant $c_2$ depends only on $\Delta_2(\phi)$, $\Delta_2(\phi^*)$, $n$ and on the ratio $|B \cap \Omega|/ |B|$. 
\end{proposition}

\begin{proof}
  We will reduce the claim to an application of Theorem~\ref{prop:Korn}. For this we extend $u$ by zero to a $W^{1,1}(B)$ function on the full ball $B$. Then with $B^+ \coloneqq \Omega \cap B$
  \begin{equation}\label{eq:korn_proof}
    \int_{B^+}\phi(|\nabla u|)\,dx =\int_B \phi(|\nabla u|)\,dx\lesssim \int_B \phi(|\nabla u - \mean{\nabla u}_B|)\,dx +|B| \phi(|\mean{\nabla u}_B|) .
  \end{equation}
  We have
  \begin{equation*}
    \phi(|\mean{\nabla u}_B|)=\fint_{B\setminus B^+}\phi(|\nabla u-\mean{\nabla u}_B|) \,dx\leq \frac{|B|}{|B\setminus B^+|}\fint_B\phi(|\nabla u-\mean{\nabla u}_B|) \,dx.
  \end{equation*}
  Combining this with \eqref{eq:korn_proof}, Theorem~\ref{prop:Korn} and Jensen's inequality we get
  \begin{align*}
    \int_{B^+} \phi(|\nabla u|)\,dx \lesssim \int_B \phi(|\eps u-\mean{\eps u}_B|)\,dx \lesssim \int_B \phi(|\eps u|)\,dx=\int_{B^+} \phi(|\eps u|)\,dx.
  \end{align*}
  This proves the claim.
\end{proof}

In similar fashion we have the following Poincaré type inequality.

\begin{proposition}[Poincaré inequality]\label{prop:Poincare}
  Let $\phi$ be an N-function with $\phi,\phi^\ast \in\Delta_2$. 
  \begin{enumerate}
  \item For all balls $B\subset \RRn$ with radius $r$ and $u\in W^{1,1}(B;\RRn)$ we have
    \begin{equation*}
      \int_B \phi\left(|u-\mean{u}_B|\right)\,dx \leq c_1 \int_B \phi(r|\nabla u|)\,dx .
    \end{equation*}
  \item Let $B \subset \RRn$ be a ball with radius~$r$ and center in $\overline{\Omega}$ such that $\Omega \cap B$ is again a Lipschitz domain. Then for all $u\in W^{1,1}(B \cap \Omega;\RRn)$ such that $u$ vanishes on $\partial \Omega \cap \overline{B}$ we have
    \begin{equation*}
      \int_{B \cap \Omega}\phi(|u|)\,dx\leq c_2 \int_{B \cap \Omega}\phi(r|\nabla u|)\,dx.
    \end{equation*}
  \end{enumerate}
  The constants $c_1$, $c_2$ depend only on $\Delta_2(\phi)$, $\Delta_2(\phi^*)$, $n$ and on the ratio $|B \cap \Omega|/ |B|$. 
\end{proposition}

The Korn and Poincaré inequality enable us to work in spaces depending on the full gradient instead of the symmetric gradient. In case of non-uniformly convex $\phi$ optimal embeddings for the spaces with the symmetric gradient can be found in~\cite{BreitCianchi2021}.

Note that Propositions~\ref{prop:Korn}, \ref{prop:KornBoundary} and~\ref{prop:Poincare} also hold for cylinders instead of balls. We will use them in this form later.

\subsection{Weak solutions}
\label{sec:weak-solutions}

For every $f \in L^{\phi^*}(\Omega)$ there exists a unique weak solution $u \in W_0^{1,\phi}(\Omega)$ of system \eqref{eq:system_in_main_thm} in the sense that for all $\psi \in W_0^{1,\phi}(\Omega)$ we have
\begin{equation}\label{eq:weak_solution}
  \int_\Omega A(\eps u) : \eps \psi \,dx = \int_\Omega f \psi \, dx.
\end{equation}
Moreover, testing with $\psi=u$ and using the \Poincare{} inequality we obtain the standard energy estimate
\begin{equation}\label{eq:energy_estimate}
  \int_\Omega \phi(|\eps u|)\,dx \leq c \int_\Omega \phi^\ast(|f|)\,dx,
\end{equation}
where $c>0$ depends only on $n$, $\Omega$, $p^+$ and $p^-$.

\subsection{Caccioppoli estimates}
\label{sec:cacc-estim}
In this subsection we derive a Caccioppoli-type inequality. Before that we recall some formulas for the symmetric gradient.

If $u\in C^1(\RRn,\RRn)$ is a vector field and $\eta\in C^1(\RRn,\setR)$ is a scalar function then
\begin{equation*}
	\eps (\eta u)=\eta \eps u + \nabla \eta \otimes_\eps u,
\end{equation*}
where for $u,v\in \RRn$ we set
\begin{equation*}
	u\otimes_\eps v \coloneqq \tfrac 12 (u \otimes v + v \otimes u) = \tfrac{1}{2} \big( u_iv_j + v_iu_j \big)_{ij}.
\end{equation*}
The set of smooth functions $u$ with $\eps u=0$ is given by the space of infinitesimal rigid motions
\begin{equation*}
  \mathcal{R}\coloneqq \lbrace x\mapsto Ax+b, \text{where } A\in \RRnn, b\in \RRn , A^T=-A\rbrace.
\end{equation*}

\begin{lemma}[Caccioppoli inequality]\label{lem:lower_order_Caccioppoli}
  Let $u$ be the weak solution of system \eqref{eq:system_in_main_thm}. We then have for every ball $B$ with radius~$r$ and $2B\subset\Omega$ the estimate
  \begin{equation}\label{eq:lower_order_Cacc_interior}
    \fint_{B} \phi(|\eps u|)\,dx \leq c_1 ~\inf _{\pi \in \mathcal{R}}\fint_{2B} \phi \left(\dfrac{|u-\pi|}{r}\right)+\phi^\ast(r|f|)\,dx.
  \end{equation}
  If $B \subset \RRn$ is a ball with center in $\overline{\Omega}$ such that $\Omega \cap B$ is again a Lipschitz domain, then
  for all $u\in W^{1,1}(B \cap \Omega;\RRn)$ such that $u$ vanishes on $\partial \Omega \cap \overline{B}$ we have
  \begin{equation}\label{eq:lower_order_Cacc_boundary}
    \fint_{B \cap \Omega} \phi(|\eps u|)\,dx \leq c_2 ~\fint_{2B \cap \Omega} \phi \bigg(\frac{|u|}{r}\bigg)+ \phi^\ast(r|f|)\,dx.
  \end{equation}
  The constants $c_1,c_2>0$ depend only on $p^-$, $p^+$, $n$ and $\abs{\Omega \cap B}/\abs{B}$.
\end{lemma}
\begin{proof}
  We begin with the case $2B \subset \Omega$. Let $\eta$ be a cut-off function with $\indicator_B \leq \eta \leq \indicator_{2B}$ and $\norm{\nabla \eta}_\infty \lesssim 1/r$. For arbitrary $\pi \in \mathcal{R}$ we use the test function $\psi = \eta ^{p^+} (u-\pi)$.  This gives us
  \begin{equation*}
    \fint_{2B} \eta^{p^+} A(\eps u): \eps u +p^+ \eta^{p^+-1} A(\eps u) :\nabla \eta \otimes_\eps (u-\pi)\,dx=\fint_{2B}\eta^{p^+}f (u-\pi)\,dx.
  \end{equation*}
  Thus, using the properties of $\eta$ we get
  \begin{equation}\label{eq:lower_cacc_estimate}
    \fint_{2B} \eta^{p^+} \phi(|\eps u|) \,dx \lesssim  \fint_{2B}\eta^{p^+-1} \phi^\prime(|\eps u|)  \dfrac{|u-\pi|}{r}\,dx +\fint_{2B} \eta^{p^+}|f||u-\pi|\,dx.
  \end{equation}
  The first integral on the right hand side of \eqref{eq:lower_cacc_estimate} can be estimated using Young's inequality \eqref{eq:young} and Lemma \ref{lem:simonenko}, giving us for every $\lambda>0$
  \begin{align*}
    \eta^{p^+-1} \phi^\prime(|\eps u|)  \dfrac{|u|}{r}
    &\leq \lambda \phi^\ast(\eta^{p^+-1} \phi^\prime(|\eps u|))+ C_\lambda \phi (|u-\pi|/r)
    \\
    &\lesssim \lambda \eta^{p^+}\phi^\ast(\phi^\prime(|\eps u|))+C_\lambda \phi(|u-\pi|/r)\\
    &\lesssim \lambda \eta^{p^+} \phi (|\eps u|)+C_\lambda \phi(|u-\pi|/r).
  \end{align*}
  Choosing $\lambda >0$ small enough, we can absorb the first term into the left hand side.  
  The second integral on the right hand side of \eqref{eq:lower_cacc_estimate} can directly be estimated using Young's inequality \eqref{eq:young}. This shows \eqref{eq:lower_order_Cacc_interior}.

  The boundary estimate~\eqref{eq:lower_order_Cacc_boundary} follows in the same way using the test function $\psi= \eta^{p^+} u$. Note that $\psi=0$ on $\partial \Omega$, since $u=0$ on $\partial \Omega$.
\end{proof}

\section{Global regularity}
\label{sec:global-regularity}

In this section we prove our main result. Since the proof for interior regularity is similar but easier than the proof of boundary regularity, we focus in our presentation on the regularity at the boundary. First, we explain how we treat the boundary by flattening.  Secondly, we show our main result for a subclass of Orlicz functions, which we call \emph{N-functions with quadratic growth}, see Definition~\ref{def:phi-quadratic}. Afterwards, we use an approximation argument to prove the general case.

\subsection{Flattening of the boundary}
\label{sec:flattening-boundary}

We flatten the boundary by using local isomorphisms which change $\Omega$ only in one direction. Similar to \cite{BerRuz17} we introduce the following notation. First we fix a point $P\in \partial \Omega$ and define what it means for $\partial\Omega$ to be locally $C^{2,1}$-regular at $P$. For this we fix a local coordinate system in which $P=0$ holds. In this coordinate system we use for every $x\in \RRn$ the notation $x^\prime \coloneqq (x_1,...,x_{n-1})$. Furthermore, we denote by $B'_r$ the $(n-1)$-dimensional ball with radius $r$ centered at $0$.

We now say that $\partial\Omega$ is locally $C^{2,1}$-regular at $P$ if there exists $0<r<1$ and a $C^{2,1}$-function $a:B^\prime _r\rightarrow (-r,r)$ with $a(0)=0$, such that $\partial\Omega$ is locally the graph of $a$, i.e.
\begin{enumerate}
	\item $\Omega_P \coloneqq \lbrace x\in \RRn: x^\prime \in B^\prime_r ,~ a(x^\prime)<x_n<a(x^\prime)+r\rbrace \subset \Omega$,
	\item $x\in \partial\Omega \cap \big(B^\prime_r\times (-r,r)\big) \iff x_n=a(x^\prime)$.
\end{enumerate}
Furthermore, we assume that the local coordinate system is chosen such that $\nabla a (0)=0$ holds. This also implies that we can make $\|a\|_{C^1(B^\prime_r)}$ arbitrarily small by choosing $r$ small enough. Finally, we say $\partial \Omega$ is $C^{2,1}$-regular, if it is locally $C^{2,1}$-regular at every point $P\in \partial \Omega$.

For $0<\lambda\leq 1$ we define the scaled sets
\begin{equation*}
	\lambda \Omega_P \coloneqq  \lbrace x\in \RRn: x^\prime \in B_{\lambda r}^\prime, ~a(x^\prime)<x_n<a(x^\prime)+\lambda r  \rbrace.
\end{equation*}
We use the notation $B^+\coloneqq \tfrac{1}{2}\Omega_P$ and $2B^+\coloneqq \Omega_P$. We call them \emph{cylinders at the boundary}.  We use the convention that we call the first $(n-1)$ directions the \emph{tangential directions} and denote them with Greek letters. The $n$-th direction we call \emph{normal direction}. Note however, that while the $n$-th unit vector $e_n$ is indeed orthogonal to $\partial \Omega$ at $P$, it is only approximately orthogonal to $\partial \Omega$ at points near $P$. This is due to the curvature of the boundary.

Next we define curved partial derivatives along the boundary. For this we fix a tangential direction $e_\alpha$ and define for a function $f:\Omega_P\rightarrow \setR$ and $h>0$ the \emph{tangential translation} via
\begin{equation*}
	f_\tau (x)\coloneqq  f\big(x^\prime +h e_\alpha, x_n+a(x^\prime+h e_\alpha)-a(x^\prime) \big).
\end{equation*}
For vector- or matrix-valued functions we apply this definition component wise. If $f\in W^{1,1}(\Omega)$ with $\support f\subset \Omega_P$ we have for almost every $x\in \Omega_P$
\begin{equation*}
	\widetilde{\partial}_\alpha f(x)\coloneqq \lim_{h\rightarrow 0}h^{-1}(f_\tau(x)-f(x))=\partial_\alpha f(x) + \partial_\alpha a(x) \partial_n f(x).
\end{equation*}
In case of a (locally) flat boundary this coincides with the partial derivative in the tangential direction $e_\alpha$. Otherwise, the second term accounts for the fact that we take the limit of small steps parallel to the boundary and not in the direction of any fixed unit vector. We now collect some formulas for the curved derivatives $\widetilde{\partial}$. 

Firstly, the curved derivatives $\widetilde{\partial}_\alpha$ do not commute with the standard partial derivatives. A straight-forward calculation reveals that for $f\in W^{1,1}(\Omega,\RRn)$ we have in $\Omega_P$
\begin{equation}\label{eq:curved_deriv_commutators}
	\begin{aligned}
		\nabla \widetilde{\partial}_\alpha f &= \widetilde{\partial}_\alpha \nabla f + \partial_\alpha \nabla a \otimes \partial_n f,\\
		\eps (\widetilde{\partial}_\alpha f) &= \widetilde{\partial}_\alpha \eps f + \partial_\alpha \nabla a \otimes_\eps \partial_n f,\\
		\divergence (\widetilde{\partial}_\alpha f) &=\widetilde{\partial}_\alpha \divergence f+\partial_\alpha \nabla a \cdot \partial_n f.
	\end{aligned}
\end{equation}
If $f,g \in W^{1,1}(\Omega)$ with $\support f \cap \support g \subset \Omega_P$ then we have the integration by parts formula
\begin{align}\label{eq:ibp_curved}
  \int_{\Omega_P} \widetilde{\partial}_\alpha f g \,dx= -\int_{\Omega_P} f\widetilde{\partial}_\alpha g\,dx.
\end{align}

Later on we will frequently use cut-off functions to localize either on a ball in the interior of $\Omega$ or on a cylinder at the boundary of $\Omega$.

Let $B$ be a ball such that $2B\subset \Omega$. Then we choose $\eta=\eta_B\in C_c^\infty (2B)$ such that
\begin{equation}\label{eq:cut_off}
	\mathbbm{1}_B \leq \eta \leq \mathbbm{1}_{2B} \quad \text{and} \quad  r\norm{\nabla \eta}_{L^\infty}+r^2\norm{\nabla^2 \eta}_{L^\infty}+r^3\norm{\nabla^3 \eta}_{L^\infty}\leq C, 
\end{equation}
where $C>0$ depends only on the dimension $n$.

For cylinders at the boundary we choose a cut-off function $\eta \in C^\infty (2B^+,[0,1])$ such that $\mathbbm{1}_{B^+}\leq \eta\leq \mathbbm{1}_{2B^+}$ and the estimate in \eqref{eq:cut_off} holds. Note that with this definition $\eta$ does not vanish on $\partial \Omega$.

\subsection{Boundary regularity for systems with quadratic growth}
\label{sec:bound-regul-line}

Our main result is proved for general uniformly convex N-functions. While it is in principle possible to work directly in this generality, it is difficult to justify some steps. It is more convenient to work with N-functions that still have some kind of quadratic growth in the sense that $\delta \leq \phi''(t) \leq \delta^{-1}$ for small~$\delta>0$. The corresponding partial differential equation then has linear growth. This allows us to work in the framework of $u \in W^{2,2}(\Omega)$ simplifying the calculations.  The estimates that we will derive will be independent of the choice of~$\delta$ and depend only on~$p^-$ and $p^+$. This allows us later by some approximation process of the N-function~$\phi$ to cover the general case of uniformly convex N-functions.

\begin{definition}[N-functions with quadratic growth]
  \label{def:phi-quadratic}
  We say that a uniformly convex N-function~$\phi$ has quadratic growth if there exists a $\delta>0$ such that for all $t\geq 0$ we have
  \begin{equation*}
    \delta \leq \phi^{\prime \prime}(t) \leq \delta^{-1}.
  \end{equation*}
\end{definition}

If $\phi$ has quadratic growth, then $\frac 12 \delta t^2 \leq \phi(t) \leq \frac 12\delta^{-1}t^2$ and hence $W^{1,\phi}(\Omega)=W^{1,2}(\Omega)$ with equivalent norms.

Note that in \eqref{eq:VA_basic1} and \eqref{eq:VA_basic2} all constants depend on $\phi$ only through $p^-$ and $p^+$. Using a limiting process we get for every $u\in W^{2,2}(\Omega)$ and almost every $x\in \Omega$ the equivalences
\begin{equation}\label{eq:AV_basic_differentiated}
\begin{aligned}
	\partial _i A(\eps u(x)):\partial_i \eps u(x)&\eqsim |\partial_i V(\eps u(x))|^2 \eqsim \phi^{\prime \prime}(|\eps u(x)|)|\partial_i \eps u(x)|^2,\\
	\text{and}\qquad |\partial_i A(\eps u(x))|&\eqsim \phi^{\prime \prime}(|\eps u(x)|)|\partial_i \eps u(x)|.
\end{aligned}
\end{equation}
where the implicit constants again only depend on $p^-$, $p^+$ and $n$. This also holds for the curved derivatives in the sense that for every cylinder $B^+$ at the boundary and $u\in W^{2,2}(B^+,\RRn)$
\begin{equation}\label{eq:A-V_for_curved}
	\widetilde{\partial}_\alpha A(\eps u): \widetilde{\partial}_\alpha \eps u \eqsim |\widetilde{\partial}_\alpha V(\eps u)|^2
\end{equation}
holds pointwise almost everywhere in $B^+$ with implicit constants depending only on $p^-$, $p^+$ and $n$.

The goal is to show the main result for N-function with quadratic growth.

\begin{proposition}\label{pro:nondegen}
  Let $\phi$ be a uniformly convex N-function with quadratic growth, see Definition~\ref{def:phi-quadratic}. Let $\Omega \subset \RRn$ be a bounded $C^{2,1}$-regular domain and $f\in W^{1,2}(\Omega) \cap W^{1,\phi^\ast}_0(\Omega)$. Then there exists a unique weak solution $u\in W_0^{1,2}(\Omega)\cap W^{2,2}(\Omega)$ to system \eqref{eq:system_in_main_thm} which fulfills $V(\eps u) \in W^{1,2}(\Omega)$ with the estimate
  \begin{equation}\label{eq:main_estimate}
    \int_{\Omega} |\nabla V(\eps u)|^2 \,dx\leq c \int_{\Omega} \phi^\ast(|f|)+\phi^\ast(|\nabla f|)\,dx,
  \end{equation}
  where $c>0$ depends only on $\Omega$, $n$ and the indices $p^-$ and $p^+$ of $\phi$.
\end{proposition}

We already know that~\eqref{eq:system_in_main_thm} admits a unique weak solution. Since $\phi$ has quadratic growth our system is of $2$-structure in the sense of~\cite[Definition 2.1]{BerRuz20}. Thus, using~\cite[Theorem 1.1]{BerRuz20} we get even more regularity, namely
\begin{equation*}
	u\in W^{1,2}_0(\Omega)\cap W^{2,2}(\Omega).
\end{equation*}
Note that this implies
\begin{equation}\label{eq:pw_equation}
	-\divergence(A(\eps u(x)))=f(x)\qquad\text{for almost every $x\in\Omega$.}
\end{equation}
In order to prove Proposition \ref{pro:nondegen}, we need to control $\partial_i V(\eps u)$ for every direction $1\leq i \leq n$. In Lemma \ref{lem:tangtang} we do this for all directions in the interior and for all tangential derivatives $\partial_\alpha$ at the boundary. The main idea is to test the weak equation \eqref{eq:weak_solution} with $\psi=\partial_{\alpha\alpha}u$. Since this is not a test function, we have to multiply it with a cut-off function.

For the normal direction $i=n$ we can not use the same approach, since $\partial_{nn}u$ does not vanish on $\partial \Omega$. Instead, we show in Lemma \ref{lem:normalnormal} that we can use the pointwise equation \eqref{eq:pw_equation} to bound the normal derivative $\partial_n V(\eps u)$ by the tangential derivatives.
Lemma \ref{lem:normaltang} is an auxiliary Lemma that we need for the proof of Lemma \ref{lem:normalnormal}. 

\begin{lemma}\label{lem:tangtang}
	Under the assumptions of Proposition \ref{pro:nondegen} we have for every ball $B$ with $2B\subset \Omega$ and every $1\leq i \leq n$
	\begin{equation}\label{eq:tangtang_interior}
		r^2\fint_{B} |\partial _i V(\eps u)|^2 \,dx \leq  c\fint_{2B} \phi^\ast(r^2|\nabla f|)+\phi(|\eps u|)\,dx,
	\end{equation}
	where $c>0$ does only depend on $n$, $p^-$, $p^+$.
	Furthermore, if $B^+$ is a cylinder at the boundary of $\Omega$, $1\leq \alpha \leq n-1$ and $\eta$ a cut-off function supported on $2B^+$ (see Section~\ref{sec:flattening-boundary}), then
	\begin{equation}\label{eq:tangtang}
		\begin{aligned}
			r^2\fint_{2B^+} \eta^2 |\partial_\alpha V(\eps u)|^2 \,dx \leq& ~  c\fint_{2B^+} \phi^\ast(r^2|\nabla f|)+\phi(|\eps u|)\,dx\\
			&+\kappa~\|\nabla a\|^2_{L^\infty}r^2\fint_{2B^+}\eta^2|\nabla V(\eps u)|^2\,dx,
		\end{aligned}
	\end{equation}
	where $c>0$ does depend on $n$, $p^-$, $p^+$ and the parametrization $a$ of the boundary, while $\kappa>0$ depends only on $n$, $p^-$, $p^+$ and is independent of $a$ .
\end{lemma}

\begin{proof}
	Let us first fix a cylinder $B^+$ at the boundary of $\Omega$. Multiplying the pointwise equation \eqref{eq:pw_equation} with $\psi \coloneqq  \widetilde{\partial}_\alpha (\eta ^2 \widetilde{\partial}_\alpha u)$ and integrating over $\Omega$ we get
	\begin{equation}\label{eq:proof_tangtang_start}
		\fint_{2B^+} \divergence A(\eps u): \widetilde{\partial}_\alpha (\eta ^2 \widetilde{\partial}_\alpha u)\,dx =\fint_{2B^+}f~ \widetilde{\partial}_\alpha (\eta ^2 \widetilde{\partial}_\alpha u) \,dx.
	\end{equation}
	We estimate the right hand side in \eqref{eq:proof_tangtang_start} using the integration by parts formula \eqref{eq:ibp_curved}, Young's inequality and Korn's inequality:
	\begin{align*}
		\fint_{2B^+}f~ \widetilde{\partial}_\alpha (\eta ^2 \widetilde{\partial}_\alpha u) \,dx&=-\fint_{2B^+}\widetilde{\partial}_\alpha f~ \eta ^2 \widetilde{\partial}_\alpha u \,dx \\
		&\lesssim r^{-2}\fint_{2B^+} \phi^\ast(r^2|\nabla f|) + \phi (|\nabla u|)\,dx\\
		&\lesssim r^{-2}\fint_{2B^+} \phi^\ast(r^2|\nabla f|)+\phi(|\eps u|) \,dx.
	\end{align*}
	The left hand side in \eqref{eq:proof_tangtang_start} can be rewritten using \eqref{eq:curved_deriv_commutators} as
	\begin{align*}
		\fint_{2B^+} \divergence A(\eps u): \widetilde{\partial}_\alpha (\eta ^2 \widetilde{\partial}_\alpha u)\,dx &=\fint_{2B^+} \widetilde{\partial}_\alpha A(\eps u):\eps (\eta ^2 \widetilde{\partial}_\alpha u)\,dx\\
		&\quad +\fint_{2B^+} \partial_\alpha \nabla a \cdot \partial_n A(\eps u) \eta^2 \widetilde{\partial}_\alpha u  \,dx\coloneqq T_1+T_2.
	\end{align*}
	We now give detailed explanation of how to estimate the term $T_2$ which appears due to the curved boundary. Subsequently, similar terms will appear and we will just refer back to the calculations here. We have for every $\lambda>0$
	\begin{equation}\label{eq:curvature_error_term1}
		\begin{aligned}
			|T_2|&\leq \|\nabla^2 a\|_{L^\infty}\fint_{2B^+} \eta^2\phi^{\prime \prime}(|\eps u|) |\partial_n \eps u||\nabla u|\,dx\\
			&\leq \dfrac{1}{2}\|\nabla^2 a\|_{L^\infty}\fint_{2B^+} \eta^2\phi^{\prime \prime}(|\eps u|) \big(\lambda|\partial_n \eps u|^2+\lambda^{-1}|\nabla u|^2\big)\,dx\\
			&\lesssim \lambda \fint_{2B^+}\eta^2|\nabla V(\eps u)|^2 \,dx+\lambda^{-1}\fint_{2B^+} \phi^{\prime\prime}(|\eps u|)|\nabla u|^2\,dx.
		\end{aligned}
	\end{equation}
	The second integral can be estimated using the pointwise estimate $|\eps u|\leq |\nabla u|$ and Korn's inequality:
	\begin{equation}\label{eq:curvature_error_term2}
		\begin{aligned}
			\fint_{2B^+} \phi^{\prime\prime}(|\eps u|)|\nabla u|^2\,dx&\leq  \fint_{2B^+} \phi^{\prime\prime}(|\nabla u|)|\nabla u|^2\,dx \\
			&\lesssim \fint_{2B^+} \phi(|\nabla u|)\,dx\lesssim \fint_{2B^+} \phi(|\eps u|)\,dx .
		\end{aligned}
	\end{equation}
	We now turn to $T_1$. We have
	\begin{align*}
		T_1&=\fint_{2B^+} \eta ^2\widetilde{\partial}_\alpha A(\eps u): \eps~\widetilde{\partial}_\alpha u\,dx+\fint_{2B^+} \widetilde{\partial}_\alpha A(\eps u):\nabla \eta^2 \otimes_\eps \widetilde{\partial}_\alpha u\,dx\\
		&=\fint_{2B^+} \eta ^2\widetilde{\partial}_\alpha A(\eps u): \widetilde{\partial}_\alpha \eps u\,dx + \fint_{2B^+} \eta ^2\widetilde{\partial}_\alpha A(\eps u): \partial_\alpha \nabla a \otimes_\eps \partial_n u\,dx\\
		&\quad + \fint_{2B^+} \widetilde{\partial}_\alpha A(\eps u):\nabla (\eta^2) \otimes_\eps \widetilde{\partial}_\alpha u\,dx\coloneqq  T_{1,1}+T_{1,2}+T_{1,3}.
	\end{align*}
	Using \eqref{eq:A-V_for_curved} we have
	\begin{equation*}
		T_{1,1}\eqsim \fint_{2B^+}  \eta^2|\widetilde{\partial}_\alpha V(\eps u)|^2\,dx.
	\end{equation*}
	We can estimate $T_{1,2}$ exactly like we estimated $T_2$ in \eqref{eq:curvature_error_term1}. Using $|\nabla \eta|\lesssim r^{-1}$ and \eqref{eq:curvature_error_term2}, the term $T_{1,3}$ can be bounded as follows.
	\begin{align*}
		|T_{1,3}|&\lesssim \fint_{2B^+} \phi^{\prime \prime}(|\eps u|) \eta |\nabla \eps u|r^{-1} |\nabla u|\,dx\\
		&\lesssim \lambda \fint_{2B^+} \eta^2 |\nabla V(\eps u)|^2 \,dx+ \lambda^{-1}r^{-2}\fint_{2B^+} \phi(|\eps u|)\,dx.
	\end{align*}
	Summarizing our findings so far (and choosing $\lambda$ small enough), we have shown
	\begin{equation*}
		\begin{aligned}
			r^2\fint_{B^+} |\widetilde{\partial}_\alpha V(\eps u)|^2 \,dx \leq& ~  c\fint_{2B^+} \phi^\ast(r|f|)+\phi^\ast(r^2|\nabla f|)+\phi(|\eps u|)\,dx\\
			&+\widetilde{c}~\|\nabla a\|_{L^\infty}\fint_{2B^+}|\nabla V(\eps u)|^2\,dx.
		\end{aligned}
	\end{equation*}
	To conclude \eqref{eq:tangtang} from here we simply note that
	\begin{equation*}
		|\partial_\alpha V(\eps u)|\leq |\widetilde{\partial} _\alpha V(\eps u)|+\|\nabla a\|_{L^\infty}|\nabla V(\eps u)|.
	\end{equation*}
	Thus, \eqref{eq:tangtang} follows. The interior estimate \eqref{eq:tangtang_interior} can be shown in the same way but with many simplifications, since we do not need to use the curved derivatives~$\widetilde{\partial}_\alpha $.
\end{proof}

\begin{lemma}\label{lem:normaltang}
	Let $B^+$ be a cylinder at the boundary and $\eta$ a cut-off function supported on $2B^+$. Let $1\leq i,j \leq n$ with $(i,j)\neq (n,n)$. Then
	\begin{equation}
		\begin{aligned}
			r^2\fint_{2B^+} \eta^2 |\partial _i A(\eps u)||\partial _j \eps u| \, dx &\leq c\fint_{2B^+} \phi^\ast(r^2|\nabla f|)+\phi(|\eps u|)\,dx\\
			&\quad +r^2\kappa \|\nabla a\|_{L^\infty}\fint_{2B^+} \eta^2|\nabla V(\eps u)|^2\,dx ,
		\end{aligned}
	\end{equation}
	where $c>0$ depends only on $p^-$, $p^+$, $n$ and the parametrization $a$ of the boundary, while $\kappa >0$ depends only on $p^-$, $p^+$, $n$ and is independent of $a$.
\end{lemma}

\begin{proof}
	Assume $i\neq n$. The case $j\neq n$ can be treated in exactly the same way. By~\eqref{eq:AV_basic_differentiated} we have for every $\lambda>0$
	\begin{align*}
		&\fint_{2B^+} \eta^2 |\partial _i A(\eps u)||\partial _j \eps u| \, dx \lesssim \fint_{2B^+} \eta ^2 \phi^{\prime \prime}(|\eps u|) |\partial_i \eps u|~|\partial _j \eps u|\, dx\\
		&\qquad\leq \dfrac{1}{2\lambda}  \fint_{2B^+} \eta ^2 \phi^{\prime \prime}(|\eps u|) |\partial_i \eps u|^2\,dx +\dfrac{\lambda}{2} \fint_{2B^+} \eta ^2 \phi^{\prime \prime}(|\eps u|) |\partial_j \eps u|^2\,dx\\
		&\qquad\lesssim  \lambda^{-1}\fint_{2B^+} \eta^2 |\partial_i V(\eps u)|^2\,dx+\lambda \fint_{2B^+} \eta^2 |\partial_j V(\eps u)|^2\,dx.
	\end{align*}
	We can now use Lemma \ref{lem:tangtang} to bound the first integral. The choice $\lambda = \|\nabla a\|_{L^\infty}$ finishes the proof.
\end{proof}

\begin{lemma}\label{lem:normalnormal}
	Let $2B^+$ be a cylinder at the boundary and $\eta$ a cut-off function supported on $2B^+$.  Then
	\begin{equation}\label{eq:normalnormal_claim}
		\begin{aligned}
			r^2\fint_{2B^+} \eta^{2} |\partial_n V(\eps u)|^2 \, dx &\leq 
			c\fint_{2B^+}\phi^\ast(r|f|)+ \phi^\ast(r^2|\nabla f|)+\phi(|\eps u|)\,dx\\
			&+r^2\kappa \|\nabla a\|_{L^\infty}\fint_{2B^+} \eta^{2} |\nabla V(\eps u)|^2\,dx,\\
		\end{aligned}
	\end{equation}
	where $c>0$ depends on $p^-$, $p^+$, $n$ and the parametrization $a$ of the boundary, while $\kappa$ depends only on $p^-$, $p^+$, $n$ and is independent of $a$.
\end{lemma}

\begin{proof}
	We adopt the method used in \cite{SerShi97}. Recall that we use the convention that the greek indices $\alpha$ and $\beta$ run only over the tangential directions $1,...,n-1$. We have
	\begin{align*}
		\fint_{2B^+} \eta^2 \partial _n A(\eps u):\partial _n \eps u ~ dx &= \sum_{\alpha,\beta} \fint_{2B^+} \eta^2 \partial _n A(\eps u)_{\alpha \beta}\partial _n \eps _{\alpha \beta}u ~ dx \\
		&+2\sum_{\alpha} \fint_{2B^+} \eta^2 \partial _n A(\eps u)_{\alpha n}\partial _n \eps_{\alpha n}u ~ dx \\
		&+\fint_{2B^+} \eta^2 \partial _n A(\eps u)_{n n}\partial _n \eps_{n n}u ~ dx\\
		&\coloneqq  S_1+S_2+S_3.		
	\end{align*}
	We start by estimating $S_2$. Using the pointwise equation \eqref{eq:pw_equation} we have
	\begin{align*}
		\dfrac{1}{2}S_2=-\sum_{\alpha \beta}\fint_{2B^+}\eta^2 \partial_\beta A(\eps u)_{\alpha \beta} \partial_n \eps_{\alpha n} u\,dx+ \sum_{\alpha} \fint_{2B^+} \eta ^2 f_\alpha \partial_n \eps_{\alpha n} u  \,dx.
	\end{align*}
	The first sum can be estimated using Lemma \ref{lem:normaltang}. To estimate the second sum we use integration by parts. Because $\eta^2 f$ vanishes on all of $\partial (2B^+)$, we get no boundary terms. Thus
	\begin{equation*}\label{eq:normalnormal_f-terms}
		\begin{aligned}
			\fint_{2B^+} \eta^2 f_\alpha \partial_n \eps_{\alpha n} u \,dx&= -\fint_{2B^+} (\eta^2 \partial_n f_\alpha +\partial_n (\eta^2) f_\alpha) \eps_{\alpha n} u \,dx \\
			&\lesssim \fint_{2B^+} |\eps u||\nabla f| + r^{-1} |\eps u| |f|  \,dx.
		\end{aligned}
	\end{equation*}
	Using Young's inequality \ref{eq:young}, we can estimate both summands. Indeed, we have
	\begin{equation*}
		|\eps u||\nabla f| \lesssim r^{-2}\big(\phi(|\eps u|)+  \phi^\ast(r^2|\nabla f|)\big),
	\end{equation*}
	and
	\begin{equation*}
		r^{-1}|\eps u||f| \lesssim r^{-2} \big(\phi(|\eps u|)+\phi^\ast(r|f|)\big).
	\end{equation*}
	We thus have shown
	\begin{equation*}
		S_2\leq \dfrac{c}{r^2}\fint_{2B^+} \phi^\ast(r|f|)+\phi^\ast(r^2|\nabla f|)+\phi(|\eps u|)\,dx .
	\end{equation*}
	We can bound $S_3$ using the same strategy. We are thus left with bounding $S_1$. For this we use the identity
	\begin{equation*}
		\partial_n \eps_{\alpha \beta}u= \partial_\alpha \eps_{\beta n}u + \partial_\beta \eps_{\alpha n}u-  \partial_{\alpha \beta} u_n.
	\end{equation*}
	This identity allows us to split $S_1$ into three terms, the first two of which can again be estimated using Lemma \ref{lem:normaltang}. This leaves us only with the third term, which we can write as
	\begin{equation}
		\sum_{\alpha \beta}\fint_{2B^+} \eta^2 \partial_n A(\eps u)_{\alpha \beta}  \partial_{\alpha\beta}u_n\,dx = A_1+A_2+A_3+A_4,
	\end{equation}
	where 
	\begin{align*}
		A_1&\coloneqq  \sum_{\alpha \beta}\fint_{2B^+} \partial_\alpha A(\eps u)_{\alpha \beta}\partial_n \eta^2 \partial_\beta u_n \,dx,\\
		A_2&\coloneqq  \sum_{\alpha \beta}\fint_{2B^+} \partial_\alpha A(\eps u)_{\alpha \beta} \eta^2 \partial_\beta \eps_{nn} u \,dx,\\
		A_3&\coloneqq  \sum_{\alpha \beta}\fint_{2B^+} A(\eps u)_{\alpha \beta}\partial_{\alpha n} \eta^2 \partial_\beta u_n \,dx,\\
		A_4&\coloneqq  \sum_{\alpha \beta}\fint_{2B^+} A(\eps u)_{\alpha \beta}\partial_{\alpha } \eta^2 \partial_\beta \eps_{nn} u \,dx .
	\end{align*}
	Here we used integration by parts to move the $\partial_n$ away from $A(\eps u)$ and the $\partial_\alpha$ away from $u_n$, and the fact that $\partial_n u_n=\eps_{nn}u$. Using Young's inequality \eqref{eq:young} we estimate
	\begin{align*}
		|A_1|&\lesssim r^{-1}\sum_\alpha \fint_{2B^+} \phi^{\prime \prime}(|\eps u|)|\partial_\alpha \eps u| |\nabla u| \,dx\\
		&\leq \dfrac{1}{2}\sum_\alpha \fint_{2B^+} \phi^{\prime \prime}(|\eps u|)\big(|\partial_\alpha \eps u|^2+r^{-2}|\nabla u|^2\big) \,dx \\
		&\lesssim \sum_\alpha\fint_{2B^+} |\partial_\alpha V(\eps u)|^2 + r^{-2}\phi^\prime(|\eps u|) \,dx,
	\end{align*}
	where we used \eqref{eq:AV_basic_differentiated} and \eqref{eq:curvature_error_term2}. For $A_2$ we have
	\begin{align*}
		|A_2|&\lesssim \sum_{\alpha \beta} \fint_{2B^+} \phi^{\prime\prime} (|\eps u|)|\partial _\alpha \eps u||\partial_\beta \eps u|\,dx\\
		&\leq \dfrac{1}{2}\sum_{\alpha \beta}\fint_{2B^+} \phi^{\prime\prime} (|\eps u|) \big(|\partial_\alpha\eps u|^2+|\partial_\beta\eps u|^2\big)\,dx.
	\end{align*}
	Both terms can be estimated with the help of Lemma \ref{lem:tangtang}. We can estimate $A_3$ as in \eqref{eq:curvature_error_term2}
	\begin{align*}
		|A_3|&\lesssim r^{-2}\fint_{2B^+} \phi^\prime(|\eps u|) |\nabla u|\,dx\\
		&\lesssim r^{-2}\fint_{2B^+} \phi(|\eps u|)\,dx.
	\end{align*}
	Finally, for $A_4$ we have
	\begin{align*}
		|A_4|&\lesssim \sum_\beta r^{-1} \fint_{2B^+} \phi^\prime(|\eps u|) |\partial_\beta \eps u|\,dx \\
		&\lesssim r^{-1}\sum_\beta \fint_{2B^+} \phi^{\prime \prime}(|\eps u|) |\partial_\beta \eps u||\eps u|\,dx\\
		&\lesssim \sum_\beta \fint_{2B^+}  \phi^{\prime\prime}(|\eps u|)\big(|\partial_\beta \eps u|^2 + |\eps u|^2 \big) \,dx\\
		&\lesssim \sum_\beta\fint_{2B^+}|\partial_\beta V(\eps u)|^2\,dx+r^{-2}\fint_{2B^+}\phi^\prime(|\eps u|) \,dx.
	\end{align*}
	Collecting our estimates for $A_1$,...,$A_4$, we get the claimed estimate for $S_1$ which together with the estimates for $S_2$ and $S_3$ completes the proof.	
\end{proof}

We are now ready to prove the main Proposition of this section.

\begin{proof}[Proof of Proposition~\ref{pro:nondegen}]
	For every cylinder $B^+$ at the boundary and every cut-off function $\eta$ supported on $2B^+$ we know by Lemma \ref{lem:tangtang} and Lemma \ref{lem:normalnormal} that there exists a $\kappa >0$ independent of the local parametrization $a$ of the boundary such that
	\begin{align*}
		r^2\fint_{2B^+} \eta^2 |\nabla V(\eps u)|^2\,dx &\leq c \fint_{2B^+}\phi^\ast(r|f|)+ \phi^\ast(r^2|\nabla f|)+\phi(|\eps u|)\,dx\\
		&\quad+n\kappa \|\nabla a\|_{L^\infty}r^2\fint_{2B^+} \eta^2 |\nabla V(\eps u)|^2\,dx.
	\end{align*}
	If we choose $B^+$ sufficiently small we can assume that
	\begin{equation}\label{eq:bound_nabla_a}
		\| \nabla a \|_{L^\infty (2B_j^+)} \leq \dfrac{1}{2n\kappa}.
	\end{equation}
	With this assumption and using $r<\diameter (\Omega)$ and Lemma~\ref{lem:simonenko} we get
	\begin{equation}\label{eq:final_bound_bdry}
		r^2\int_{2B^+} \eta^2 |\nabla V(\eps u)|^2\,dx \leq \widetilde{c}~ \int_{2B^+}\phi^\ast(|f|)+ \phi^\ast(|\nabla f|)+\phi(|\eps u|)\,dx.
	\end{equation}
	The same bound holds for all interior balls $B$ with $2B\subset \Omega$.
	
	We now cover $\Omega$ by choosing a finite collection of balls $(B_i)_{i\in I}$ resp. cylinders at the boundary $(B_j^+)_{j\in J}$ such that
	\begin{enumerate}
		\item The $B_i$ together with the $B_j^+$ cover $\Omega$,
		\item $2B_i\subset \Omega$ resp. $2B_j^+\subset \Omega$ holds for all $i\in I$ resp. $j\in J$,
		\item There exists a natural number $C=C(n)$ such that every $x\in \Omega$ is contained in at most $C$ of the $B_i$ or $B_j^+$.
	\end{enumerate}
	Additionally we assume that \eqref{eq:bound_nabla_a} holds for all of the $B_j^+$. We denote by $\eta_i$ resp. $\eta_j$ the cut-off function corresponding to $B_i$ resp. $B_j^+$. We have
	\begin{align*}
		\int_\Omega |\nabla V(\eps u)|^2\,dx&\leq \sum_i \int_{B_i}|\nabla V(\eps u) |^2 \,dx + \sum_j \int_{B_j^+}|\nabla V(\eps u)|^2 \,dx \\
		&\leq \sum_i \int_{2B_i}r_i^2\eta_i^2|\nabla V(\eps u) |^2 \,dx + \sum_j \int_{2B_j^+}r_j^2\eta_j^2|\nabla V(\eps u)|^2 \,dx .
	\end{align*}
	Using \eqref{eq:final_bound_bdry} and the finite overlap of our covering we get
	\begin{equation*}
		\int_\Omega |\nabla V(\eps u)|^2\,dx\leq c \int_{\Omega}\phi^\ast(|f|)+ \phi^\ast(|\nabla f|)+\phi(|\eps u|)\,dx.
	\end{equation*}
	Together with the energy estimate \eqref{eq:energy_estimate} this implies the claimed bound.
\end{proof}

\subsection{Passage to the limit}
\label{sec:passage-limit}

Proposition~\ref{pro:nondegen} already contains the results of Theorem \ref{thm:orlicz} but is restricted to N-functions with quadratic growth. In this section we will show how to approximate every uniformly convex N-function by those with quadratic growth and how to pass to the limit.

For this let $\phi$ be a uniformly convex N-function. For every $0 \leq \delta^- \leq \delta^+ \leq \infty$ and $\delta\coloneqq (\delta^-,\delta^+)$ we define the truncated N-functions $\phi_\delta=\phi_{(\delta^-,\delta^+)}$ via
\begin{equation*}
  \phi_\delta^\prime(t)=\phi_{(\delta^-,\delta^+)}'(t)\coloneqq  \dfrac{\phi^\prime(\delta^-\lor t\land \delta^+)}{\delta^-\lor t \land \delta^+}t
\end{equation*}
and $\phi_\delta (t)\coloneqq \int_0^t \phi^\prime _\delta(s)\,ds$.  These functions were introduced in~\cite{DieForTom20}. See the Appendix for more details. For example 
by Lemma~\ref{lem:index_of_phi_delta} we know that $\phi_\delta$ is a uniformly convex N-function with quadratic growth.

Accordingly, we define for $Q\in \RRnn$
\begin{equation}
	A_\delta (Q)\coloneqq \phi_\delta^\prime(|Q|)\dfrac{Q}{|Q|} \qquad \text{and} \qquad V_\delta(Q)\coloneqq  \sqrt{\phi^\prime_\delta(|Q|)|Q|}\dfrac{Q}{|Q|}.
\end{equation}
We can apply Proposition~\ref{pro:nondegen} only to systems where $f$ is contained in $W^{1,2}(\Omega)$. Thus we also need to approximate $f$. For this we perform a \emph{Lipschitz truncation}, i.e. we approximate $f$ by functions $f_\delta \in W^{1,\infty}(\Omega)$ which agree with $f$ outside of a small set.

As an auxiliary tool we need the Hardy-Littlewood maximal operator
\begin{equation*}
	(Mf)(x):= \sup_{r>0}\fint_{B_r(x)} |f|\,dy.
\end{equation*}
To approximate~$f$ we will use the following approximation result.

\begin{lemma}[Lipschitz truncation]
  \label{lem:liptruncation}
  Let $\Omega \subset \RRn$ be a Lipschitz domain. Let $\phi$ be a uniformly convex N-function. Let $v \in W^{1,\phi}_0(\Omega)$. Then for all $\lambda >0$ there exists $T_\lambda v \in W^{1,\infty}_0(\Omega)$ with the following properties
  \begin{enumerate}
  \item \label{itm:liptruncation1} $\set{v \neq T_\lambda v} \subset \set{M(\nabla v)>\lambda}$.
  \item $\norm{T_\lambda v}_{L^\phi(\Omega)} \lesssim \norm{v}_{L^\phi(\Omega)}$ and $\int_\Omega \phi(\abs{T_\lambda v})\,dx \lesssim \int_\Omega \phi(\abs{v})\,dx$.    
  \item $\norm{\nabla T_\lambda v}_{L^\phi(\Omega)} \lesssim \norm{\nabla v}_{L^\phi(\Omega)}$ and $\int_\Omega \phi(\abs{\nabla T_\lambda v})\,dx \lesssim \int_\Omega \phi(\abs{\nabla v})\,dx$.    
  \item $\abs{\nabla T_\lambda v} \lesssim \lambda \indicator_{\set{M(\nabla v)>\lambda}} + \abs{\nabla v}_{\Omega \setminus \set{M(\nabla v)> \lambda}} \leq \lambda$ almost everywhere.
  \item $\norm{\nabla (v-T_\lambda v)}_{L^\phi(\Omega)} \lesssim \norm{\nabla v}_{L^\phi(\Omega)}$.
  \item 
    $\int_\Omega \phi(\abs{\nabla (v-T_\lambda v)})\,dx \lesssim \int_\Omega \indicator_{\set{M(\nabla v)>\lambda}}
    \phi(\abs{\nabla v})\,dx$.
  \item $T_\lambda v \to v$ in $W^{1,\phi}_0(\Omega)$ as $\lambda \to \infty$.
  \end{enumerate}
  The hidden constants only depend on the indices $p^-$ and $p^+$ of $\phi$ and $\Omega$.
\end{lemma}

\begin{proof}
  The uniform convexity of~$\phi$ ensures that $1<p^- \leq p^+ < \infty$. Now, the proof follows exactly as in~\cite{DieKreuSueli13,BulDieSch16}. However, the boundedness of the maximal operator~$M$ in $L^p$ has to be replaced by the corresponding results on Orlicz spaces, see~\cite{KraRut61}. 
\end{proof}

We are now ready to prove Theorem \ref{thm:orlicz}.
\begin{proof}[Proof of Theorem \ref{thm:orlicz}]
  In order to apply Proposition~\ref{pro:nondegen} we approximate $f$ by applying the above Lipschitz truncation $f_\delta \coloneqq T_\lambda f \in W^{1,\infty}_0(\Omega)$ on the level $\lambda=c\,\phi^\prime(\delta^+)$, so that $\abs{\nabla f_\delta} \leq \phi'(\delta^+)$.  The level~$\lambda = c\phi^\prime(\delta^+)$ is chosen in such a way that the truncation of $\phi$ at $\delta^{+}$ (near $\infty$) does not see the difference between $\nabla f_\delta$ and $\nabla f$.
	
  By Proposition~\ref{pro:nondegen} we know that for every $0<\delta^- \leq \delta^+<\infty$ there exists a unique solution $u_\delta \in W_0^{1,\phi_\delta}(\Omega)$ such that
	\begin{equation*}
		\begin{aligned}
			-\divergence (A_\delta (\epsilon u_\delta))&= f_\delta\qquad \text{in }\Omega \\
			u_\delta &= 0 \qquad \text{on }\partial \Omega 
		\end{aligned}
	\end{equation*}
  and (also using the Orlicz-\Poincare{} inequality Propositon~\ref{prop:Poincare})
	\begin{equation}\label{eq:approx_uniform_bound}
		\|V_\delta(\eps u_\delta)\|^2_{W^{1,2}(\Omega)}\leq  c\, \int_\Omega
    \phi^\ast_\delta (|\nabla f_\delta|)\,dx \coloneqq c K_\delta,
  \end{equation}
  where $c$ is independent of $\delta$. In a moment, it will turn out that $K_\delta$ can be bounded independently of $\delta$. Using $\abs{\nabla f_\delta} \leq \phi'(\delta^+)$, Lemma \ref{lem:index_of_phi_delta} \ref{itm:index_of_phi_delta5} and Lemma~\ref{lem:liptruncation} we have
  \begin{align*}
    K_\delta
    &= \int_\Omega (\phi^\ast)_{\phi^\prime(\delta),\phi^\prime(\delta^{+})}(|\nabla f_\delta|)\,dx
    \\
    &= \int_\Omega (\phi^\ast)_{\phi^\prime(\delta),\infty}(|\nabla f_\delta|)\,dx
    \\
    &\lesssim \phi(\delta^-) + \int_\Omega \phi^\ast(|\nabla f_\delta|)\,dx
    \\
    &\lesssim \phi(\delta^-) + \int_\Omega \phi^\ast(|\nabla f|)\,dx.
  \end{align*}
	The first term vanishes as $\delta^-\rightarrow 0$. Thus, combining this with \eqref{eq:approx_uniform_bound}, we have  uniformly in $\delta$
	\begin{equation}\label{eq:approx_bound}
    	\|V_\delta(\eps u_\delta)\|^2_{W^{1,2}(\Omega)}\lesssim \phi(\delta^-)+ \int_{\Omega} \phi^\ast (|\nabla f|)\,dx.
  \end{equation}
  In the following we assume $\delta^-\leq 1$ and $\delta^+ \geq 1$ and are interesting in the limit $\delta^- \to 0$ and $\delta^+\to \infty$. Since $\norm{V_\delta(\epsilon u_\delta)}_{L^2(\Omega)}$ is uniformly bounded, so is~$\norm{\epsilon u_\delta}_{L^{\phi_\delta}(\Omega)}$.  Setting $s\coloneqq p^- \land 2$, Lemma \ref{lem:uniform_embedding} implies that the $u_\delta$ are uniformly bounded in $W^{1,s}(\Omega)$.
	Thus, after passing to a (non relabeled) subsequence we have for some $u^\ast\in W^{1,s}(\Omega)$
	\begin{equation}\label{eq:convergence1}
		\begin{aligned}
			u_\delta &\rightharpoonup u^\ast \qquad\text{weakly in $W^{1,s}(\Omega)$,}\\
			u_\delta &\rightarrow u^\ast \qquad\text{strongly in $L^s$.}
		\end{aligned}
	\end{equation}
	In the following we will show that $u^\ast$ fulfills the regularity estimate \eqref{eq:bound_in_main_thm} and then that it is indeed the weak solution to the system \eqref{eq:system_in_main_thm}.
	
	From \eqref{eq:approx_bound} we can assume that there exists $Q\in W^{1,2}(\Omega)$ such that
	\begin{equation}\label{eq:convergence2}
		\begin{aligned}
		V_\delta(\eps u _\delta) &\rightharpoonup Q \qquad \text{weakly in $W^{1,2}(\Omega)$,}\\
		V_\delta(\eps u _\delta) &\rightarrow Q \qquad \text{strongly in $L^2(\Omega)$,}\\
		V_\delta(\eps u _\delta) &\rightarrow Q \qquad \text{pointwise almost everywhere in $\Omega$.}
	\end{aligned}
      \end{equation}
	Using Lemma \ref{lem:continuities}, the pointwise convergence in \eqref{eq:convergence2} implies that $\eps u_\delta =V_\delta^{-1}\circ V_\delta (\eps u_\delta)$ converges pointwise almost everywhere to $V^{-1}(Q)$. Since we also know from \eqref{eq:convergence1} that $\eps u_\delta \rightharpoonup \eps u^\ast$ weakly in $L^s$, we have $V(\eps u^\ast)=Q$. We can now use the first convergence in \eqref{eq:convergence2} together with the uniform bound \eqref{eq:approx_bound} to conclude
	\begin{equation}\label{eq:conclusion}
		\|V(\eps u^\ast)\|^2_{W^{1,2}(\Omega)} \leq c(n,\Omega,p^-,p^+) \int_\Omega \phi^\ast (|\nabla f|)\,dx.
	\end{equation}
	We are left to show that $u^\ast$ is the weak solution of the system \eqref{eq:system_in_main_thm}. Using Lemma \ref{lem:continuities} we know that $A_\delta (\eps u_\delta) \rightarrow A(\eps u^\ast)$ almost everywhere. Note that
	\begin{equation*}
		\int_\Omega (\phi _\delta)^\ast (|A_\delta (\eps u_\delta)|) \,dx \eqsim \int_\Omega \phi_\delta (\abs{\eps u_\delta}) \,dx \leq c(\Omega ,f).
	\end{equation*}
	Thus, by Lemma~\ref{lem:uniform_embedding}, $\epsilon u_\delta$ is uniformly bounded in $L^{p^- \wedge 2}(\Omega)$ and $A_\delta (\eps u_\delta)$ uniformly bounded in $L^{(p^+)' \wedge 2}(\Omega)$. This together with the almost everywhere convergence imply convergence in $L^1$ (for a subsequence). Thus, $A_\delta (\eps u _\delta) \rightarrow A(\eps u^\ast)$ in $L^1(\Omega)$. This implies for all $\psi \in C_c^\infty (\Omega)$
	\begin{align*}
		\skp{A(\eps u^\ast)}{\eps \psi} &= \lim_{\delta \rightarrow 0} \skp{A_\delta (\eps u_\delta)}{\eps \psi}\\
		&=\lim_{\delta \rightarrow 0} \skp{f_\delta}{\psi}\\
		&=\skp{f}{\psi}.
	\end{align*}
	This shows that $u^\ast$ indeed solves \eqref{eq:system_in_main_thm} in the weak sense. Therefore, \eqref{eq:conclusion} gives us the claimed bound \eqref{eq:bound_in_main_thm}, which finishes the proof.
\end{proof}

Finally, we prove Theorem~\ref{thm:p-growth}.

\begin{proof}[Proof of Theorem~\ref{thm:p-growth}]
	We restrict our proof to the case $\delta =0$, the case $\delta >0$ can be treated in the same way. The claimed estimate follows immediately from Theorem \ref{thm:orlicz}. We are thus left to prove the claimed Sobolev regularity of $u$. In the case $p\geq 2$ we note that $V(\eps u) =  |\eps u|^{\tfrac{p-2}{2}}\eps u$ implies
  \begin{equation*}
    \eps u=|V(\eps u)|^{\tfrac{2}{p}}\dfrac{V(\eps u)}{|V(\eps u)|}.,	
  \end{equation*}
  Using \cite[Proposition 4.4]{BalDieWei20}\footnote{The result in~\cite{BalDieWei20} is proven for balls but extends easily to Lipschitz domains using the result of~\cite{Shvartsman2006}.}%
    we get that $\eps u \in W^{2/p,p}$ and thus by Korn's inequality $u \in W^{1+2/p,p}$.
  
  For $p<2$ we first note that by Sobolevs embedding we have
  \begin{equation*}
    V(\eps u)\in W^{1,2}(\Omega) \implies V(\eps u)\in L^{\tfrac{2n}{n-2}}(\Omega) \implies \eps u \in L^{\tfrac{np}{n-2}}.
  \end{equation*}
  Setting $r\coloneqq \tfrac{np}{n+p-2}$ we get
  \begin{align*}
    \int_\Omega |\nabla \eps u|^r\,dx&=\int_\Omega \Big(|\nabla \eps u|^2|\eps u|^{p-2} \Big)^{r/2}|\eps u|^{\tfrac{(2-p)r}{2}} \,dx\\ 
    &\lesssim \Big(\int_\Omega  |\nabla V(\eps u)|^2 \,dx\Big)^{r/2}\Big(  \int_\Omega|\eps u|^{\tfrac{np}{n-2}}\,dx\Big)^{\tfrac{2-r}{2}}.
  \end{align*}
  We thus have $\nabla \eps u \in L^r(\Omega)$ which implies $u \in W^{2,r}(\Omega)$.
\end{proof}

\appendix
\section{Uniformly convex N-functions}
\label{sec:unif-conv-orlicz-appendix}

In this Appendix we collect some properties of uniformly convex N-functions (also called uniformly convex Young functions), see Definition~\ref{def:uniform_convexity}. We will carefully ensure that all constants (implicit and explicit) in this section depend on the Orlicz function $\phi$ only through the indices $p^-$ and $p^+$.

\subsection{Basic properties}

Let us collect a few basic properties of uniformly convex N-functions.
In the following let $p'=\frac{p}{p-1}$ denote the conjugate exponent of~$p$.
\begin{lemma}\label{lem:uniformly_convex_dual}
	If $\phi$ is a uniformly convex N-function with indices $p^-$ and $p^+$, then $\phi^\ast$ is also uniformly convex with indices $(p^+)^\prime$ and $(p^-)^\prime$.
\end{lemma}
\begin{proof}
	This is shown in \cite[Lemma 29]{DieForTom20} and \cite[Lemma 6.4]{DieRuz07}.
\end{proof}

The following lemma shows that the indices of uniform convexity provide an upper and lower bound for the Simonenko indices. See~\cite{FioKRb97} for an overview on indices on Orlicz spaces.
\begin{lemma}[Simonenko indices]
  \label{lem:simonenko}
  Let $\phi$ be uniformly convex with indices~$p^-$ and $p^+$. Then
  \begin{align*}
    p^- \leq \inf_{t>0} \frac{\phi'(t) t}{\phi(t)} \leq \sup_{t>0} \frac{\phi'(t) t}{\phi(t)}\leq p^+.
  \end{align*}
  In particular, 
  \begin{align*}
			\min \set{t^{p^+},t^{p^-}}\phi(s)&\leq \phi(ts)\leq \max  \set{t^{p^+},t^{p^-}} \phi(s) \quad \text{for all $s,t \geq 0$.}
  \end{align*}
\end{lemma}
\begin{proof}
	The upper bound with $p^+$ follows from 
	\begin{equation}\label{eq:appendixaux}
		\phi^\prime(t)t=\int_0^t \phi^{\prime\prime}(s)s+\phi^\prime(s)\,ds\leq \int_0^t p^+\phi^\prime(s) \,ds=p^+\phi(t).
	\end{equation}
	The lower bound with $p^-$ follows for analogously.

  The rest of the claim follows by the theory on Simonenko indices, but let us include a short proof. If $t\geq 1$, then
	\begin{align*}
		\ln(\phi(st))-\ln(\phi(s))&=\int_{1}^{t}\dfrac{s\phi^{\prime}(\lambda s)}{\phi(\lambda s)}\,d\lambda \leq \int_{1}^{t}\lambda^{-1}p^+\,d\lambda=p^+ \ln(t).
	\end{align*}
  We thus have $\phi(st)\leq t^{p^+}\phi(s)$ for $t \geq 1$. The other bounds follow similarly.
\end{proof}

\begin{corollary}[$\Delta_2$ and Young's inequality]
  Let $\phi$ be a uniformly convex N-function with indices $p^-$ and $p^+$. Then $\phi$ and $\phi^*$ satisfy the $\Delta_2$-condition, i.e. for all $t \geq 0$
  \begin{align}
    \label{eq:delta2}
    \begin{aligned}
      \phi(2t) &\leq 2^{p^+} \phi(t),
      \\
      \phi^*(2t) &\leq 2^{(p^-)'} \phi^*(t).
    \end{aligned}
  \end{align}
  Moreover, for all $s,t\geq 0$ and $\delta \in (0,1]$
  \begin{align}
    \label{eq:young}
    \begin{aligned}
      st&\leq \delta^{1-{p^+}} \phi(s)+\delta \phi^\ast (t),
      \\
      st&\leq \delta \phi(s)+\delta^{1-(p^-)'} \phi^\ast (t).
    \end{aligned}
  \end{align}
\end{corollary}

\begin{proof}
  The first part follows from Lemma~\ref{lem:simonenko}. The second part follows by Young's inequality and  Lemma~\ref{lem:simonenko}, i.e.
  \begin{align*}
    s\,t = \delta (s/\delta) t \leq \delta \big( \phi(s/\delta) + \phi^*(t)\big) \leq \delta^{1-p^+}\phi(s) + \delta \phi^*(t).
  \end{align*}
  The other estimate follows by replacing~$\phi$ with $\phi^*$.
\end{proof}

For N-functions~$\phi$ one has for $t\geq 0$
\begin{align*}
  \phi(t) \leq \phi'(t) t \leq \phi(2t).
\end{align*}
As a consequence one obtains for $t \geq 0$
\begin{align*}
	 \phi^*\big(\phi'(t)\big)\leq t \phi'(t)\leq \phi^\ast \big(2\phi' (t)\big).
\end{align*}
Hence, if $\phi$ is uniformly convex with indices~$p^-$ and $p^+$, then for $t \geq 0$
\begin{align}
  \label{eq:orlicz_basic2}
  2^{-(p^-)'}\phi(t) \leq \phi^*\big(\phi'(t)\big)  \leq 2^{p^+} \phi(t).  
\end{align}
Recall that by~\eqref{eq:defAV}
\begin{align*}
	A(P)= \phi^\prime(|P|)\dfrac{P}{|P|}\qquad\text{and}\qquad V(P)= \sqrt{\phi^\prime(|P|)|P|}\dfrac{P}{|P|}.
\end{align*}

\begin{remark}
  \label{rem:AV}
  As $A$ is induced by~$\phi$, $V$ is induced by~$\psi$, where $\psi$ is the N-function defined by
  \begin{align*}
    \psi'(t) \coloneqq \sqrt{\phi'(t)\,t}.
  \end{align*}
  One easily calculates
  \begin{align*}
    \frac{\psi''(t)t}{\psi'(t)} &=\frac 12 \frac{\phi''(t)t}{\phi'(t)} + \frac 12. 
  \end{align*}
  This proves that $\psi$ is a uniformly convex N-function and the indices of $\psi$ and $\phi$ are related by the formula %
  \begin{align*}
    p^\pm_\psi &= 1 + \frac{p^\pm_\phi}{2}.
  \end{align*}
\end{remark}

The quantities $A$, $V$ and $\phi$ are linked by the following important equivalence.

\begin{lemma}\label{lem:Hammer}
	Let $\phi$ be a uniformly convex N-function. Then for all $P,Q \in \RRnn$
	\begin{align*}
		(A(P)-A(Q)):(P-Q)&\eqsim |V(P)-V(Q)|^2 \eqsim \phi''(|P|+|Q|)|P-Q|^2,
	\end{align*}
	where the implicit constants depend only on the indices $p^-$ and $p^+$.
\end{lemma}

\begin{proof}
	This Lemma goes back to \cite[Lemma 3]{DieEtt08}.
\end{proof}

\subsection{Truncated N-functions}

For  $0 \leq \delta^-\leq \delta^+\leq \infty$ and $\delta = (\delta^-,\delta^+)$ we define the truncated N-function $\phi_\delta$ by
\begin{equation*}
  \phi_\delta'(t)=\phi_{(\delta^-,\delta^+)}^\prime(t)\coloneqq   \frac{\phi^\prime(\delta^-\lor t\land \delta^+)}{\delta^-\lor t \land \delta^+}\,t.
\end{equation*}
These functions were first introduced in~\cite{DieForTom20}. If $\delta^+=\infty$, then they are related to the shifted N-functions introduced in \cite{DieEtt08}.
\begin{lemma}\label{lem:index_of_phi_delta}
  Let $\phi$ be a uniformly convex N-function with indices $p^-$ and $p^+$ and $0<\delta^-\leq \delta^+<\infty$. Then
  \begin{enumerate}
  \item \label{itm:index_of_phi_delta1} $\phi_\delta$ is a uniformly convex N-function with quadratic growth, see Definition~\ref{def:phi-quadratic}.
  \item \label{itm:index_of_phi_delta2} The lower index of $\phi_\delta$ is $p^- \land 2$, and the upper index of $\phi_\delta$ is $p^+\lor 2$.
  \item \label{itm:index_of_phi_delta3} $\phi^\prime(t)=\phi^\prime_\delta(t)$ and $\phi^{\prime\prime}(t)=\phi^{\prime\prime}_\delta(t)$ for all $t\in (\delta^-,\delta^+)$.
  \item \label{itm:index_of_phi_delta5} We have
    \begin{equation*}
      (\phi_{\delta})^\ast =  (\phi^\ast)_{\phi^\prime (\delta^-),\phi^\prime(\delta^+)}.
    \end{equation*}
  \item \label{itm:index_of_phi_delta4} For $t\in [0,\delta^+]$ we have
    \begin{equation*}
      \abs{\phi(t) - \phi_\delta(t)} \leq 2^{p^+-1}\phi(\delta^-).
    \end{equation*}
  \end{enumerate}
\end{lemma}

\begin{proof}
  By the definition of $\phi_\delta$ we have
  \begin{equation*}
    \phi_\delta^{\prime \prime}(t)=
    \begin{cases}
      \tfrac{\phi^\prime(\delta^-)}{\delta^-} &\qquad \text{for }t\leq \delta^-,
      \\
      \phi^{\prime\prime}(t) &\qquad \text{for }t\in (\delta^-,\delta^+),
      \\
      \tfrac{\phi^\prime(\delta^+)}{\delta^+} &\qquad \text{for }t\geq \delta^+.
    \end{cases}
  \end{equation*}
  and thus
  \begin{equation*}
    \dfrac{\phi_\delta^{\prime \prime}(t)t}{\phi_\delta^\prime(t)}+1=
    \begin{cases}
      2 &\qquad\text{for } t\leq \delta^-,
      \\
      \dfrac{\phi^{\prime\prime}(t)t}{\phi^\prime(t)}+1 &\qquad\text{for }t\in (\delta^-,\delta^+),
      \\
        2 &\qquad\text{for }t\geq \delta^+.
      \end{cases}
  \end{equation*}
  These observations readily imply \ref{itm:index_of_phi_delta1}--\ref{itm:index_of_phi_delta3}.

  To see \ref{itm:index_of_phi_delta5} we note that
  \begin{alignat*}{2}
    [(\phi^\ast)_{\phi^\prime (\delta^-),\phi^\prime(\delta^+)}]^\prime(t) &=\dfrac{\phi^{\ast\prime}(\phi^\prime(\delta^-)\lor t\land \phi^\prime(\delta^+))}{\phi^\prime(\delta^-)\lor t \land \phi^\prime(\delta^+)}\,t
    \\
    &=\dfrac{\delta^-\lor \phi^{\ast \prime}(t)\land \delta^+}{\phi^\prime(\delta^-)\lor t \land \phi^\prime(\delta^+)}\,t = \big((\phi_\delta)'\big)^{-1}(t) = \big((\phi_\delta)^*\big)'(t).
  \end{alignat*}
  This proves~\ref{itm:index_of_phi_delta5}.
  
  It remains to prove~\ref{itm:index_of_phi_delta4}.
  If $t \in [0,\delta^+]$, then
  \begin{align*}
    \abs{\phi(t) - \phi_\delta(t)}
    &\leq \int_0^t \abs{\phi'(s) - \phi_\delta'(s)}\,ds
    \\
    &\leq \tfrac 12 \delta^- \phi'(\delta^-) \leq \tfrac 12 \phi(2 \delta^-) \leq 2^{1-p^+} \phi(\delta^-).
  \end{align*}
  This proves the last claim.
\end{proof}

\begin{lemma}\label{lem:uniform_embedding}
	Let $\phi$ be a uniformly convex N-function with indices $p^-$ and $p^+$ and $0<\delta^- \leq 1 \leq \delta^+<\infty$. Then 
	\begin{equation*}
    L^{p^-}(\Omega) \cap
    L^{p^+}(\Omega) \cap L^2(\Omega) \embedding L^{\phi_\delta}(\Omega)
    \embedding L^{p^-}(\Omega) +
    L^{p^+}(\Omega) + L^2(\Omega)
	\end{equation*}
  with embedding constants uniformly bounded with respect to~$\delta$.
\end{lemma}

\begin{proof}
  It follows from Lemma~\ref{lem:simonenko} and Lemma~\ref{lem:index_of_phi_delta}~\ref{itm:index_of_phi_delta2} that
  \begin{align*}
			\min \set{t^{p^+},t^{p^-},t^2}\phi_\delta(1)&\leq \phi_\delta(t)\leq \max  \set{t^{p^+},t^{p^-},t^2} \phi_\delta(1) \quad \text{for all $t \geq 0$.}
  \end{align*}
  Note that using $1 < 2 \wedge p^- \leq 2 \vee p^+ < \infty$ we calculate
  \begin{align*}
    \phi_\delta (1)\eqsim 1\cdot \phi_\delta'(1)=\phi'(1)\eqsim \phi(1).
  \end{align*}
  Hence,
  \begin{align*}
    \min \set{t^{p^+},t^{p^-},t^2}\phi(1)&\lesssim \phi_\delta(t)\lesssim \max  \set{t^{p^+},t^{p^-},t^2} \phi(1) \quad \text{for all $t \geq 0$.}
  \end{align*}
  This proves the claim.
\end{proof}

For $Q \in \RRnn$ define
\begin{equation}
	A_\delta (Q)\coloneqq \phi_\delta^\prime(|Q|)\dfrac{Q}{|Q|} \qquad \text{and} \qquad V_\delta(Q)\coloneqq  \sqrt{\phi^\prime_\delta(|Q|)|Q|}\dfrac{Q}{|Q|}.
\end{equation}

\begin{remark}
  \label{rem:AVdelta}
  In the sense of Remark~\ref{rem:AV} $A_\delta$ is induced by $\phi_\delta$  and $V_\delta$ is induced by~$\psi_\delta$ (the truncation of~$\psi$) and $\psi'_\delta(t) = \sqrt{\phi_\delta'(t) t}$. Note that also the $\psi_\delta$ are uniformly convex N-functions by Lemma~\ref{lem:index_of_phi_delta}.
\end{remark}

The following continuity is inspired by~\cite{BerDieRuz10}.

\begin{lemma}\label{lem:continuities}
	Let $\phi$ be a uniformly convex N-function. Then 
  $(\delta^-,\delta^+,t)\mapsto \phi^\prime_\delta (t)$ is continuous on $[0,1]\times [1,\infty] \times \setR^{\geq 0}$. Furthermore, the mapping $(\delta^-,\delta^+,Q)\mapsto A_{\delta}(Q)$ is continuous $[0,1]\times [1,\infty] \times \RRnn$. The same holds true if we replace $A_\delta$ with $A^{-1}_\delta$, $V_\delta$ or $V_\delta^{-1}$.
\end{lemma}

\begin{proof}
  Recall that $\phi'$ is continuous and that
  \begin{align*}
    \phi_\delta'(t)=\phi_{(\delta^-,\delta^+)}'(t)= \phi'(\delta^-\lor t\land \delta^+) \frac{t}{\delta^-\lor t \land \delta^+}.
  \end{align*}
  This proves that $(\delta^-,\delta^+,t) \mapsto \phi_\delta'(t)$ is continuous on $[0,1] \times [1,\infty] \times [0,\infty)$. Analogously
  \begin{align*}
    A_\delta(Q) = \phi'(\delta^-\lor \abs{Q}\land \delta^+) \frac{Q}{\delta^-\lor \abs{Q} \land \delta^+}
  \end{align*}
  shows that $A_\delta(Q)$ is continuous. Using $(\phi_\delta)^* = (\phi^*)_{\phi'(\delta)}$ it follows that $A_\delta^{-1}(Q)$ is continuous. The continuity of~$V_\delta(Q)$ and $V_\delta^{-1}(Q)$ follows by Remark~\ref{rem:AVdelta}.
\end{proof}

\printbibliography
\end{document}